\documentclass{amsart}
\usepackage{graphicx} 

\usepackage{amsfonts}%
\usepackage{amsmath}%
\usepackage{amssymb}%
\usepackage{graphicx}
\usepackage{cite}
\usepackage{geometry}
\usepackage{enumerate}
\usepackage[table]{xcolor}
\usepackage{multirow}
\usepackage{enumitem}
\usepackage{nicematrix}
\usepackage{hhline}
\usepackage{booktabs}
\usepackage{tabularx}
\usepackage{mathrsfs}
\usepackage{wrapfig}
\usepackage{caption}
\usepackage{subcaption}
\usepackage[justification=centering]{caption}

\usepackage{chemformula}

\providecommand{\U}[1]{\protect \rule{.1in}{.1in}}
\newtheorem{theorem}{Theorem}
\theoremstyle{plain}

\newtheorem{corollary}{Corollary}

\newtheorem{example}{Example}

\newtheorem{proposition}{Proposition}

\numberwithin{equation}{section}
\begin{document}
\title[Korovkin type theorem via Power Series Methods]{On Korovkin-Type Theorems Including Exponential Test Functions on infinite intervals through Power Series Convergence}
\author{Dilek Söylemez$^1$}
\author{Mehmet Ünver$^2$}

\footnote{Department of Mathematics, Faculty of Science, Selcuk University,
Selcuklu, 42003 Konya, Turkiye, dsozden@gmail.com\\ $^{2}$Department of Mathematics, Faculty of Science, Ankara University,
Ankara, Turkiye, munver@ankara.edu.tr}


\begin{abstract}
Approximation theory has long been concerned with the development of positive linear operators that effectively approximate classes of functions. Among the most well-known results in this area are Korovkin-type approximation theorems, which provide simple and elegant criteria for convergence by testing only on a small set of functions. Motivated by these classical results and their extensions, we focus on versions that preserve exponential functions and incorporate modern summability techniques. In this paper, we establish Korovkin-type theorems that preserve exponential functions by employing power series convergence and a special case thereof. By considering approximation through Borel-type power series convergence via integral summability, we provide an alternative framework that applies in cases where classical convergence or ordinary Borel convergence fails, and we offer a comparative analysis of the corresponding theorems. We also present illustrative examples in which the classical results fail, while the proposed approach remains applicable. In addition, the rate of convergence is analyzed through the modulus of continuity.
\end{abstract}
\maketitle

\section{Introduction}

In the classical Korovkin theorem \cite{altomare, korovkin}, established on compact intervals, the convergence of a sequence of positive linear operators is guaranteed by testing the functions $1$, $\xi$, and $\xi^{2}$. This highlights the central role of polynomials as a test set within compact domains. However, exponential functions, characterized by their unique growth rates, cannot be exactly represented by any finite polynomial on the real axis. Despite this, they can be effectively approximated in calculus and are widely used across various research fields. To extend the scope of Korovkin-type results beyond compact intervals, Boyanov and Veselinov \cite{boyanov-veselinov} established a theorem that preserves the behavior of exponential functions. This refinement is particularly advantageous on infinite intervals, as it enables the study of positive linear operators whose action is naturally tied to exponential growth, ensuring that the sequence of operators not only approximates simple polynomials but also more complex functions like exponentials.


The approximation results obtained via classical convergence for operators preserving exponential functions were studied by several mathematicians. For example, Acar et al. \cite{acar-aral-gonska} studied Szász–Mirakyan operators that preserve constants and $e^{2ax}$ for $a>0$ in the framework of the Mediterranean Journal of Mathematics. Later, Acar et al. \cite{acar-aral-} introduced a different modification of Szász-Mirakyan-type operators that fix exponential functions and defined a new weighted modulus of continuity, while Acu and Tachev \cite{acu-tachev} proposed another new variant of the Szász–Mirakyan operator which preserve $e^{ax}$ for $a\in%
\mathbb{R}
.$ Further contributions include Aral et al. \cite{aral-inoan-rasa}, who analyzed the weighted approximation properties of Szász–Mirakyan operators preserving exponential functions, and Gupta and Aral \cite{gupta-aral}, who considered Szász–Mirakyan Kantorovich type operators preserving $e^{-x}$. In addition, Phillips operators preserving exponential functions have also been investigated in\cite {gupta-tachev}. Baskakov-type operators preserving exponential functions have also been investigated in \cite{ozsarac-acar, ovgu-ali, gupta-acu }.


The study of exponential Korovkin-type results was studied by means of some summability methods and significantly stronger results were obtained in them. Mursaleen and Mohiuddine
\cite{musaleen-moh} presented Korovkin-type theorems, using almost convergence and statistical convergence. Y\i ld\i z
et.al \cite{yildiz} introduced Korovkin-type theorems,\ via power series statistical convergence. Since
summability methods are highly valuable when the ordinary convergence of
positive linear operators to a function fails, they have been extensively
employed in approximation theory (see e.g., \cite{dinar-kadim, duman2007, gadjiev-orhan, soylemez-un2021, soy-guven, sumbul-yildirim, ulucay-un, unver-orhan2019, yildiz-demirci-2025 })

In the works of Holhoş \cite{holhos, holhos2018}, the  rate of approximation of functions on infinite intervals by means of positive linear operators was studied, yielding quantitative convergence estimates that are crucial for practical applications in functional analysis. These contributions demonstrate the ongoing development of Korovkin-type approximation theory, which has gradually expanded to encompass more complex functions and diverse mathematical settings, thereby highlighting the central role of preserving exponential functions within this framework. In particular, Holhoş \cite{holhos} established rates of convergence involving exponential and rational functions by considering specific Korovkin subsets. These subsets are given by $\{ f_{j} \}_{j=0}^{2}$ and $\{ g_{j} \}_{j=1}^{2}$, where $f_{j}(\xi)=e^{-j\xi}$ and $g_{j}(\xi)=\left(\frac{\xi}{1+\xi}\right)^{j}$. Such results have further enriched the theoretical foundations and practical scope of positive linear operators.


In this paper, we address existing gaps in Korovkin-type approximation theory by focusing on operators that preserve exponential functions beyond the scope of classical convergence. While earlier results mainly concentrate on compact intervals and polynomial test functions, many positive linear operators naturally arise on unbounded domains where exponential functions play a central role. Moreover, classical convergence often fails for certain operator sequences, making it necessary to employ summability techniques as an alternative. Motivated by these challenges, we first establish Korovkin-type theorems via power series convergence and its special case, the Borel method. We then advance the study by considering approximation through Borel-type power series convergence via integral summability, which provides a stronger framework in situations where neither classical nor ordinary Borel convergence applies. This combined approach not only broadens the applicability of Korovkin’s theorem to infinite intervals but also ensures more robust convergence results. Furthermore, we evaluate the convergence rate by employing the modulus of continuity and support our theoretical findings with explicit examples and graphical illustrations.


The outline of the paper is given as follows. In Section \ref{sec2}, we recall the preliminary concepts and notations required for the study. Section \ref{sec3} presents a Korovkin-type theorem preserving exponential functions via power series convergence; as a special case, we focus on the Borel method and provide an example that fails to converge in the classical sense. In Section \ref{sec4}, we establish a Korovkin-type theorem through approximation by Borel-type power series convergence via integral summability, and compare it with the previous results, illustrating the discussion with an example. Section \ref{sec5} is devoted to the rates of convergence, which are derived by means of the modulus of continuity under the proposed summability framework. Finally, Section \ref{sec6} concludes the paper with a summary of findings.


\section{Preliminaries}\label{sec2}
In this section, we recall some basic concepts and notations that will be frequently used throughout the paper. We begin by outlining the framework of power series convergence and its regularity conditions, highlighting the classical Abel and Borel convergence methods as particular cases. These tools provide the summability background necessary for our subsequent Korovkin-type approximation results. In addition, we introduce certain function spaces, namely $C[0,\infty)$ and $B[0,\infty)$, together with their subspaces consisting of functions having finite limits at infinity. Within this setting, we establish a novel proposition showing that $B_{\ast}[0,\infty)$, equipped with the supremum norm, forms a Banach space. This result extends the functional analytic foundations available for Korovkin-type theorems on unbounded intervals and will play a crucial role in the development of our main results.

We start with the power series convergence and its regularity  \cite{boos}. Consider a real sequence $\{\rho_{m}\}$ satisfying $\rho_{0}>0$ and $\rho_{m}\geq0$ for all $m\geq1$, and let
\[
\rho(y)=\sum_{m=0}^{\infty}\rho_{m}y^{m}
\]
denote the associated power series, whose radius of convergence is $R$ with $0<R\leq\infty$. If we have
\[
\lim_{0<y\rightarrow R^{-}}\frac{1}{\rho(y)}\sum_{m=0}^{\infty}x_{m}\rho
_{m}y^{m}=L,
\]
then we say that the number sequence $x=\left(  x_{m}\right)  $ is convergent in the sense of
power series method $P_{\rho}$, shortly $P$.

It is important to recall a fundamental criterion for the regularity of power series methods. Regularity ensures that the method is consistent with ordinary convergence, and thus plays a central role in summability theory.

\begin{theorem}\cite{boos}
A power series method $P$ is regular if and only if, for every $m=0,1,...$,
\[
\lim_{0<y\rightarrow R^{-}}\frac{\rho_{m}y^{m}}{\rho(y)}=0.
\]

\end{theorem}

Various Korovkin-type approximation theorems based on power series convergence
have been studied in the literature, particularly in \cite{tas-atlihan, tas-yrd-atl2018, yurd-2016}.  Power series convergence includes
Abel and Borel convergence. In the context of Abel convergence, relevant
Korovkin-type theorems are provided in \cite{unver, unver2}
Additional results concerning different function spaces and applications for different operators can be found in
\cite {atlihan-unver,sakaoglu-unver, soy-unver2017, soylemez-unver-2020,soy-2025}. 

Abel and Borel convergence, as specific types of
power series convergence, are briefly reviewed below.

Assume that $\rho_{m}=1$ for any $m=0,1,...$, in this case $R=1$ and $\rho(y)=\dfrac{1}{1-y}.$ Thus the
power series convergence reduces to the Abel convergence. Let $x=\{x_{m}\}$
be a real sequence. If the series
\begin{equation*}
\sum_{m=0}^{\infty}x_{m}y^{m} \label{1d1}%
\end{equation*}
is convergent for any $y\in(0,1)$ and%
\[
\lim_{0<y\rightarrow1^{-}}(1-y)\sum_{m=0}^{\infty}x_{m}y^{m}=L,
\]
then $x$ is said to be Abel convergent to real number $L$ \cite{boos, powell}.

Assume that $\rho_{m}=\dfrac{1}{m!},$ in this case $R=\infty$ and $\rho(y)=e^{y}.$ Thus the power series convergence reduces to the
Borel convergence. Let $x=(x_{m})$ be a real sequence. If the series
\begin{equation*}
\sum_{m=0}^{\infty}x_{m}y^{m}%
\end{equation*}
is convergent for any $y>0$ and
\[
\lim_{0<y\rightarrow \infty}e^{-y}\sum_{m=0}^{\infty}\frac{x_{m}}{m!}%
y^{m}=L,
\]
then $x$ is said to be the Borel convergent to the real number $L$. \cite{boos, powell}

Now, we recall that $C[0,\infty)$ is the space of all real-valued and
continuous functions on $[0,\infty)$ and $B[0,\infty)$ is the space of all
real valued and bounded functions on $[0,\infty)$. These spaces are Banach
spaces with the norm

\begin{equation}
||\phi||_{\infty}=\sup_{\xi \ge 0}|\phi(\xi)|.\label{1f}%
\end{equation}
The subspaces of these spaces that ensure $\displaystyle\lim_{\xi \rightarrow \infty}\phi(\xi)$ exists
are denoted by $C_{\divideontimes}[0,\infty)$ and $B_{\divideontimes}[0,\infty)$,
respectively. It is well known that $C_{\divideontimes}[0,\infty)$ is a Banach space with the norm (\ref{1f}) (see, e.g., \cite{holhos, holhos2018}) According to the proposition below, $B_{\divideontimes}[0,\infty),$ is also a Banach space.

\begin{proposition}
\label{pro1}$B_{\divideontimes}[0,\infty)$ is a Banach space with the norm
(\ref{1f})
\end{proposition}

\begin{proof}
Linearity of $B_{\divideontimes}[0,\infty)$ is immediate. Since $B[0,\infty)$
is a Banach space, it suffices to show that $B_{\divideontimes}[0,\infty)$ is
closed with respect to the norm \eqref{1f}. Let $\phi$ belong to the closure of
$B_{\divideontimes}[0,\infty)$. Then there exists a sequence
$\{\phi_{n}\}\subset B_{\divideontimes}[0,\infty)$ such that
$\|\phi_{n}-\phi\|_{\infty}\to 0$. Thus, $\phi$ is bounded. As $\{\phi_{n}\}$ is Cauchy, for any $\varepsilon>0$ there exists $n_{0}$ such that
\[
\|\phi_{n}-\phi_{m}\|_{\infty}<\varepsilon
\]
whenever  $n,m\geq n_{0}$. For each $n$, let $\displaystyle\lim_{\xi\to\infty}\phi_{n}(\xi)=L_{n}$. Then, for $n,m\geq n_{0}$,
\[
|L_{n}-L_{m}|=\lim_{\xi\to\infty}|\phi_{n}(\xi)-\phi_{m}(\xi)|
   \leq \|\phi_{n}-\phi_{m}\|_{\infty}<\varepsilon.
\]
Thus $\{L_{n}\}$ is a  real Cauchy sequence and converges to some
finite limit $L$. Since $\phi_{n}\to\phi$ uniformly and $L_{n}\to L$, there exists $n_{1}$ such that
\[
\|\phi_{n}-\phi\|_{\infty}<\tfrac{\varepsilon}{3}
\quad\text{and}\quad
|L_{n}-L|<\tfrac{\varepsilon}{3}
\]
whenever $n\geq n_{1}$.
Fix $n\geq n_{1}$. As $\displaystyle\lim_{\xi\to\infty}\phi_{n}(\xi)=L_{n}$, there exists $M>0$ with
\[
|\phi_{n}(\xi)-L_{n}|<\tfrac{\varepsilon}{3} 
\]
whenever $\xi\geq M$.
Hence, for all $\xi\geq M$, we have
\[
|\phi(\xi)-L|
  \leq |\phi(\xi)-\phi_{n}(\xi)| + |\phi_{n}(\xi)-L_{n}| + |L_{n}-L|
  < \varepsilon.
\]
Therefore $\lim_{\xi\to\infty}\phi(\xi)=L$, showing that
$\phi\in B_{\divideontimes}[0,\infty)$. Consequently,
$B_{\divideontimes}[0,\infty)$ is closed.
\end{proof}

\section{Approximation via power series convergence}\label{sec3}

In this section, we introduce a Korovkin type theorem on infinite intervals via power series convergence. Due to the fact that many
linear positive operators are defined on unbounded intervals, Theorem
\ref{theorem2} can be useful in this regard. Before we proceed, we state a Korovkin-type theorem involving classical convergence that include exponential test functions, originally established in \cite{boyanov-veselinov}.

\begin{theorem}
\label{boyanov}Let $\{ \Re_{m}\}$ be a sequence of linear positive operators
from $C_{\divideontimes}[0,\infty)$ into $C_{\divideontimes}[0,\infty)$
satisfying the condition
\[
\lim_{m\rightarrow \infty}\left \Vert \Re_{m}  \phi_{\nu}
-\phi_{\nu}\right \Vert _{\infty}=0
\]
for any $\nu=0,1,2$, then for any $\phi \in C_{\divideontimes}[0,\infty)\ $ we
have
\[
\lim_{m\rightarrow \infty}\left \Vert \Re_{m}  \phi   -\phi
\right \Vert _{\infty}=0,
\]
where $\phi_{\nu}\left(  t\right)  =e^{-\nu t}$ for $\nu=0,1,2$.
\end{theorem}

We are now prepared to establish the following Korovkin-type approximation theorem.

\begin{theorem}
\label{theorem2}Let $\left \{  \Re_{m}\right \}  $ be a sequence of linear
positive operators from $C_{\divideontimes}[0,\infty)\ $ into
$B_{\divideontimes}[0,\infty)\ $such that $%
{\displaystyle \sum \limits_{m=0}^{\infty}}
\left \Vert \Re_{m}(\phi_{0})\right \Vert y^{m}<\infty$ for any $0<y<R$. If 
\begin{equation}
\lim_{0<y\rightarrow R^{-}}\frac{1}{p(y)}\left \Vert \sum \limits_{m=0}^{\infty
}(  \Re_{m}  \phi_{\nu}  -\phi_{\nu})
y^{m}\right \Vert _{\infty}=0 \label{2a}%
\end{equation}
for any\ $\nu=0,1,2$, then for any $\phi \in C_{\divideontimes
}[0,\infty)\ $ we have
\begin{equation}
\lim_{0<y\rightarrow R^{-}}\frac{1}{p(y)}\left \Vert \sum \limits_{m=0}^{\infty
}\left(  \Re_{m}  \phi   -\phi \right)  y^{m}\right \Vert _{\infty
}=0. \label{2a1}%
\end{equation}

\end{theorem}

\begin{proof}
Suppose that $\phi \in C_{\divideontimes}[0,\infty),$ then there exists
$H>0$ such that $\left \vert \phi(\xi)\right \vert \leq H$ for $\xi \in
\lbrack0,\infty).$ So we get%
\begin{equation}
\left \vert \phi \left(  t\right)  -\phi \left(  \xi \right)  \right \vert \leq2H,
\label{2b}%
\end{equation}
for $\xi \in \lbrack0,\infty).$ It is easy to see that for a given
$\varepsilon>0$ there exists $\eta>0$ such that
\[
\left \vert \phi \left(  t\right)  -\phi \left(  \xi \right)  \right \vert
<\varepsilon
\]
whenever $\left \vert e^{-t}-e^{-\xi}\right \vert <\eta$. Thus, we reach to%
\begin{equation}
\left \vert \phi \left(  t\right)  -\phi \left(  \xi \right)  \right \vert
\leq \varepsilon+\frac{2H}{\eta^{2}}\left(  e^{-t}-e^{-\xi}\right)  ^{2}.
\label{5d}%
\end{equation}
Using $\left(
\ref{5d}\right)  $ in the operators, we have
\begin{align}
\left \vert \Re_{m}\left(  \phi \left(  t\right)  ;\xi \right)  -\phi \left(
\xi \right)  \right \vert  &  \leq \left(  \Re_{m}\left(  \left \vert \phi \left(
t\right)  -\phi \left(  \xi \right)  \right \vert ;\xi \right)  \right)
\label{5d1}\\
&  +H\left \vert \left(  \Re_{m}\left(  \phi_0(t);\xi \right)  -\phi_0(\xi)\right)  \right \vert
.\nonumber
\end{align}
Thus, we obtain for all $y\in(0,R)$ that
\begin{align*}
\frac{1}{p(y)}\left \vert \sum \limits_{m=0}^{\infty}\left(  \Re_{m}\left(
\phi \left(  t\right)  ;\xi \right)  -\phi \left(  \xi \right)  \right)
p_{m}y^{m}\right \vert  &  \leq \frac{1}{p(y)}\sum \limits_{m=0}^{\infty}\left(
\Re_{m}\left(  \left \vert \phi \left(  t\right)  -\phi \left(  \xi \right)
\right \vert ;\xi \right)  \right)  p_{m}y^{m}\\
&  +H\frac{1}{p(y)}\left \vert \sum \limits_{m=0}^{\infty}\left(  \Re_{m}\left(
\phi_{0}(t);\xi \right)  -\phi_{0}(\xi) \right)  p_{m}y^{m}\right \vert.
\end{align*}
From (\ref{2a}), it follows that the second term of the last sum tends to zero as $y \to R^-$.
Since for any non-negative integer $m$ the operator $\left \{  \Re_{m}\right \}  $  is linear and positive  we obtain from (\ref{2a}) that
\begin{align*}
\frac{1}{p(y)}\sum \limits_{m=0}^{\infty}\left(  \Re_{m}\left(  \left \vert
\phi \left(  t\right)  -\phi \left(  \xi \right)  \right \vert ;\xi \right)
\right)  p_{m}y^{m} &  \leq \left(  \varepsilon+\frac{2H}{\eta^{2}}\right)
\frac{1}{p(y)}\left \vert \sum \limits_{m=0}^{\infty}\left(  \Re_{m}\left(
1;\xi \right)  -1\right)  p_{m}y^{m}\right \vert \\
&  +\frac{4H}{\eta^{2}}\frac{1}{p(y)}\left \vert \sum \limits_{m=0}^{\infty
}\left(  \Re_{m}\left(  \left(  e^{-t}\right)  ;\xi \right)  -e^{-\xi}\right)
p_{m}y^{m}\right \vert \\
&  +\frac{2H}{\eta^{2}}\frac{1}{p(y)}\left \vert \sum \limits_{m=0}^{\infty
}\left(  \Re_{m}\left(  e^{-2t};\xi \right)  -e^{-2\xi}\right)  p_{m}%
y^{m}\right \vert 
\end{align*}
for all $y\in \left(  0,R\right)  $. Then we obtain
\begin{align*}
\frac{1}{p(y)}\left \Vert \sum \limits_{m=0}^{\infty}\left(  \Re_{m}\phi
-\phi \right)  p_{m}y^{m}\right \Vert _{\infty} &  \leq K\left \{  \frac{1}%
{p(y)}\left \Vert \sum \limits_{m=0}^{\infty}\left(  \Re_{m}\phi_{2}-\phi
_{2}\right)  p_{m}y^{m}\right \Vert _{\infty}\right.  \\
&  +\frac{1}{p(y)}\left \Vert \sum \limits_{m=0}^{\infty}\left(  \Re_{m}\phi
_{1}-\phi_{1}\right)  p_{m}y^{m}\right \Vert _{\infty}\\
&  +\left.  \frac{1}{p(y)}\left \Vert \sum \limits_{m=0}^{\infty}\left(  \Re
_{m}\phi_{0}-\phi_{0}\right)  p_{m}y^{m}\right \Vert _{\infty}\right \}  ,
\end{align*}
where $K=\max \left \{  \varepsilon+\dfrac{2H}{\eta^{2}},\dfrac{4H}{\eta^{2}%
}\right \}  .$ Hence, we arrive that
\[
\lim_{0<y\rightarrow R^{-}}\frac{1}{p(y)}\sum \limits_{m=0}^{\infty}\left(
\Re_{m}\left(  \left \vert \phi \left(  t\right)  -\phi \left(  \xi \right)
\right \vert ;\xi \right)  \right)  p_{m}y^{m}=0,
\]
which ends the proof.
\end{proof}

In the following, we give a consequence of Theorem \ref{theorem2} \ via Borel convergence.

\begin{corollary}
\label{Borel corollary copy(1)}Let $\left \{  \Re_{m}
\right \} $ be a sequence of linear positive operators from $C_{\divideontimes
}[0,\infty)\ $to $B_{\divideontimes}[0,\infty)\ $such that $%
{\displaystyle \sum \limits_{m=0}^{\infty}}
\left \Vert \Re_{m}\phi_{0}\right \Vert y^{m}<\infty$ for any $y>0$. If%
\begin{equation}
\lim_{0<y\rightarrow \infty}e^{-y}\left \Vert \sum \limits_{m=0}^{\infty}(
\Re_{m}\left(  \phi_{\nu}\right)  -\phi_{\nu})  \frac{y^{m}}%
{m!}\right \Vert _{\infty}=0, \label{6a}%
\end{equation}
for all $\nu=0,1,2$, then for any $\phi \in C_{\divideontimes
}[0,\infty)\ $we have
\[
\lim_{0<y\rightarrow \infty}e^{-y}\left \Vert \sum \limits_{m=0}^{\infty}\left(
\Re_{m}  \phi   -\phi \right)  \frac{y^{m}}{m!}\right \Vert
_{\infty}=0.
\]

\end{corollary}

The following example demonstrates the existence of a sequence of positive linear 
operators that satisfies the conditions of Theorem~\ref{theorem2}, 
but does not fulfill the conditions of Theorem~\ref{boyanov}.

\begin{example}
Consider the sequence $\left \{  \sigma_{m}\right \}  $ defined by%
\[
\sigma_{m}=\left \{
\begin{array}
[c]{rll}%
1 & , & \text{if }m=2k\\
-1 & , & \text{if }m=2k+1,\text{ }k=0,1,...
\end{array}
\right.
\]
and the sequence of linear positive operators $\left \{  \mathcal{L}%
_{m}\right \}  $ defined for $m=0,1,...$ by
\[%
\mathcal{L}%
_{m}\phi=(1+\sigma_{m})S_{m+1}\phi \text{,}%
\]
where $\left \{  S_{m}\right \}  $ is the sequence of the Sz\'{a}sz-Mirakjan
operators defined from $C_{\divideontimes}[0,\infty)$ to itself defined by%
\begin{equation}
S_{m}(\phi(t);\xi)=e^{-m\xi}\sum \limits_{k=0}^{\infty}\phi \left(  \frac{k}%
{m}\right)  \frac{(m\xi)^{k}}{k!}\label{szasz}%
\end{equation}
for $\xi \in \lbrack0,\infty)$. We know that $\left \{  S_{m}\right \}  $ satisfies (see;
\cite{holhos2018})%
\begin{equation}
S_{m}(\phi_{0};\xi)=1,\label{SZASZ1}%
\end{equation}%
\begin{equation}
S_{m}(\phi_{1};\xi)=e^{m\xi(e^{-\frac{1}{m}}-1)},\label{szasz-2}%
\end{equation}
and%
\begin{equation}
S_{m}(\phi_{2};\xi)=e^{m\xi(e^{-\frac{2}{m}}-1)}.\label{szasz-3}%
\end{equation}
Thus, we get%
\[%
\mathcal{L}%
_{m}(\phi_{0}\left(  t\right)  ;\xi)=(1+\sigma_{m}),
\]%
\[%
\mathcal{L}%
_{m}(\phi_{1}\left(  t\right)  ;\xi)=\left(  1+\sigma_{m})\right)  e^{\left(
m+1\right)  \xi(e^{-\frac{1}{m+1}}-1)},
\]
and%
\[%
\mathcal{L}%
_{m}(\phi_{2}\left(  t\right)  ;\xi)=\left(  1+\sigma_{m}\right)  e^{\left(
m+1\right)  \xi(e^{-\frac{2}{m+1}}-1)}\text{.}%
\]
As $\left \{  \sigma_{m}\right \}  $ diverges $\left \{  \mathcal{L}_{m}\right \}
$ does not satisfy the conditions of Theorem \ref{boyanov}. On the other hand,
since $\left \{  \sigma_{m}\right \}  $ is a Borel null sequence $\left \{
\mathcal{L}%
_{m}\right \}  $ holds (\ref{6a}). From Theorem \ref{theorem2} we obtain%
\[
\lim_{0<y\rightarrow \infty}e^{-y}\left \Vert \left(
\mathcal{L}%
_{m}(\phi)-\phi \right)  \frac{y^{m}}{m!}\right \Vert _{\infty}=0.
\]

\end{example}

\section{Approximation by Borel-type power series convergence via integral
summability}\label{sec4}

In this section, we employ integral summability methods to establish a Korovkin-type theorem that preserves exponential functions. This result enables us to handle sequences on infinite intervals that fail to converge either in the classical sense or under a Borel-type power series method.

Let $P$ be a non-polynomial Borel-type power series method (that is,
$R=\infty$) and let $\left \{  \Re_{m}\right \}  $ be a sequence of positive
linear operators from $C_{\divideontimes}[0,\infty)$ to $B_{\divideontimes
}[0,\infty)$ such that
\begin{equation}
M=\sup_{y>0}\frac{1}{p(y)}\sum_{m=0}^{\infty}\left \Vert \Re_{m}\phi
_{0}\right \Vert _{\infty}p_{m}y^{m}<\infty.\label{1}%
\end{equation}
Then for any $y>0$ the operator $\mathcal{V}_{P,\Re}^{y}:C_{\divideontimes
}[0,\infty)\rightarrow B_{\divideontimes}[0,\infty)$ defined by
\begin{equation}
\mathcal{V}_{P,\Re}^{y}\phi \left(  \xi \right)  =\frac{1}{p(y)}\sum
_{m=0}^{\infty}\Re_{m}(\phi(y);\xi)p_{m}y^{m}\label{3}%
\end{equation}
is a positive linear operator.

Note that $\mathcal{V}_{P,\Re}^{y}$ is well defined by (\ref{1}) for each
$y>0$. Since the pointwise limit of a sequence of measurable functions (which,
in this case, converges uniformly) is itself measurable, it follows that for
any $\xi \in \lbrack0,\infty)$, the function $\left(  \mathcal{V}_{P,\Re}%
^{y}\phi \right)  \left(  \xi \right)  $ is measurable with respect to the
variable $y$. Consequently, we can define a new positive linear operator by
employing the concept of integral summability. We now introduce a new positive linear operator constructed via integral summability.  
Let \( \mathcal{F} \) be a non-negative and regular integral summability method.  
For any \( s \in [0,\infty) \), the operator \( \mathcal{F}_{P,\Re}^{s} \) is defined by
\begin{equation}
\left(  \mathcal{F}_{P,\Re}^{s}\phi \right)  (\xi)={\displaystyle \int
\limits_{0}^{\infty}}\mathcal{F}(s,y)\left(  \mathcal{V}_{P,\Re}^{y}%
\phi \right)  \left(  \xi \right)  dy\label{22}%
\end{equation}
and it is again a positive linear operator. Since $\mathcal{F}$ is
non-negative and regular, it follows from \cite{connor} that
\begin{equation}
\lim_{s\rightarrow \infty}{\displaystyle \int \limits_{0}^{\infty}}%
\mathcal{F}(s,y)dy=1\text{.}\label{21}%
\end{equation}

Before presenting the Korovkin type theorem, it is necessary to provide the
following proposition.

\begin{proposition}
Let $\mathcal{F}$ be a non-negative regular integral summability method$,$ $P$
be a non-polynomial Borel-type power series method and $\left \{  \Re
_{m}\right \}  $ be a sequence of positive linear operators from
$C_{\divideontimes}[0,\infty)$ to $B_{\divideontimes}[0,\infty)$ that
satisfies (\ref{1}). Then for any $\phi \in C_{\divideontimes}[0,\infty)$ and
$s\in \mathbb{[}0,\infty \mathbb{)}$ we have $%
\mathcal{F}%
_{p,\Re}^{s}$ is a positive linear operator from $C_{\divideontimes}%
[0,\infty)$ to $B_{\divideontimes}[0,\infty).$
\end{proposition}

\begin{proof}
Linearity is clear from the definition involving integral and sums. For
positivity, we assume $\phi \in C_{\divideontimes}[0,\infty)$ with $\phi \geq0$.
Since the operators $\left \{  \Re_{m}\right \}  $ are positive, the kernel
$\mathcal{F}(s,y)$ is non-negative, and $p_{m},y^{m}\geq0$, every term in the
integral and series definition of the operator is non-negative, that is,
\[
\left(  \mathcal{F}_{P,\Re}^{s}\phi \right)  (\xi)=\int_{0}^{\infty}%
\mathcal{F}(s,y)\left(  \mathcal{V}_{P,\Re}^{y}\phi \right)  \left(
\xi \right)  dy=\int_{0}^{\infty}\mathcal{F}(s,y)\frac{1}{p(y)}\sum
_{m=0}^{\infty}\Re_{m}(\phi(y);\xi)p_{m}y^{m}dy\geq0.
\]
Thus, $\mathcal{F}_{P,\Re}^{s}$ is a positive operator. Define $b_{m}%
=\displaystyle \lim_{\xi \rightarrow \infty}\Re_{m}(\phi(y);\xi)$. From (\ref{1})
we get for any $y>0$ that%
\[
\sum_{m=0}^{\infty}\left \vert b_{m}\right \vert p_{m}y^{m}\leq \frac{\left \Vert
\phi \right \Vert _{\infty}}{p(y)}\sum_{m=0}^{\infty}\left \Vert \Re_{m}\phi
_{0}\right \Vert p_{m}y^{m}<\infty \text{,}%
\]
which yields that the series $\displaystyle \sum_{m=0}^{\infty}b_{m}p_{m}y^{m}$
converges for any $y>0$. Again from (\ref{1}) the series in (\ref{3})
converges uniformly in $\xi$ for any fixed $y$. So we have
\begin{align}
\lim_{\xi \rightarrow \infty}\left(  \mathcal{V}_{P,\Re}^{y}\phi \right)  (\xi)
&  =\frac{1}{p(y)}\sum_{m=0}^{\infty}\lim_{\xi \rightarrow \infty}\Re_{m}%
(\phi(y);\xi)p_{m}y^{m}\nonumber \\
&  =\frac{1}{p(y)}\sum_{m=0}^{\infty}b_{m}p_{m}y^{m}\nonumber \\
&  <\infty \label{6b}%
\end{align}
for all $y>0$. Moreover, we have
\begin{align}
\left \vert \left(  \mathcal{V}_{P,\Re}^{y}\phi \right)  \left(  \xi \right)
\right \vert  &  \leq \frac{1}{p(y)}\sum_{m=0}^{\infty}\left \vert \Re_{m}%
(\phi(y);\xi \right \vert )p_{m}y^{m}\nonumber \\
&  \leq \frac{\left \Vert \phi \right \Vert _{\infty}}{p(y)}\sum_{m=0}^{\infty
}\left \Vert \Re_{m}\phi_{0}\right \Vert p_{m}y^{m}\nonumber \\
&  \leq M\left \Vert \phi \right \Vert _{\infty}.\label{6b1}%
\end{align}
From (\ref{6b1}) we have
\begin{equation}
\left \vert \mathcal{F}(s,y)\left(  \mathcal{V}_{P,\Re}^{y}\phi \right)
\right \vert \leq M\left \Vert \phi \right \Vert _{\infty}\mathcal{F}%
(s,y).\label{6c}%
\end{equation}
Since the integrand is dominated by the integrable function $M\left \Vert
\phi \right \Vert \mathcal{F}(s,\cdot)$ in (\ref{6c}) and the pointwise limit
(\ref{6b}) exists, we apply the Dominated Convergence Theorem to interchange
the limit and the integral and have
\begin{align}
\lim_{\xi \rightarrow \infty}\left(  \mathcal{%
\mathcal{F}%
}_{p,\Re}^{s}\phi \right)  (\xi) &  =%
{\displaystyle \int \limits_{0}^{\infty}}
\mathcal{F}%
(s,y)\lim_{\xi \rightarrow \infty}\left(  \mathcal{V}_{P,\Re}^{y}\phi \right)
\left(  \xi \right)  dy\nonumber \\
&  =%
{\displaystyle \int \limits_{0}^{\infty}}
\mathcal{%
\mathcal{F}%
}(s,y)\frac{1}{p(y)}\sum_{m=0}^{\infty}b_{m}p_{m}y^{m}dy.\label{101}%
\end{align}
The existence of this limit is guaranteed by the application of DCT, as the
integral of the dominating function is finite. Finally, we show that
$\mathcal{F}_{P,\Re}^{s}\phi \in B_{\divideontimes}[0,\infty)$ by proving its
boundedness. Using (\ref{6c}), we write
\begin{align*}
\left \vert \left(
\mathcal{F}%
_{P,\Re}^{s}\phi \right)  (\xi)\right \vert  &  \leq%
{\displaystyle \int \limits_{0}^{\infty}}
\mathcal{F}%
(s,y)\left \vert \left(  \mathcal{V}_{P,\Re}^{y}\phi \right)  \left(
\xi \right)  \right \vert dy\\
&  =M\left \Vert \phi \right \Vert _{\infty}%
{\displaystyle \int \limits_{0}^{\infty}}
\mathcal{F}%
(s,y)dy\\
&  \leq M\left \Vert \phi \right \Vert _{\infty}K\\
&  <\infty \text{,}%
\end{align*}
where $K=\displaystyle \sup_{s\geq0}\displaystyle \int_{0}^{\infty}%
\mathcal{F}(s,y)dy<\infty$ is guaranteed by the regularity of $\mathcal{F}$.
The boundedness of $\left(  \mathcal{F}_{P,\Re}^{s}\phi \right)  (\xi)$ for all
$\xi,s\in \lbrack0,\infty)$ implies $\mathcal{F}_{P,\Re}^{s}\phi \in
B_{\divideontimes}[0,\infty)$. This completes the proof.
\end{proof}

If a sequence of positive linear operators fails to converge under a standard
Borel-type power series method, the following theorem provides an alternative
convergence criterion using an integral summability method.

\begin{theorem}
\label{theorem1}Let $P$ be a non-polynomial Borel-type power series method and
let $\left \{  \Re_{m}\right \}  $ be a sequence of positive linear operators
from $C_{\divideontimes}[0,\infty)$ to $B_{\divideontimes}[0,\infty)$ that
satisfies (\ref{1}). If for any $\nu=0,1,2$%
\begin{equation}
\lim_{s\rightarrow \infty}\left \Vert
\mathcal{F}%
_{P,\Re}^{s}\phi_{\nu}-\phi_{\nu}\right \Vert _{\infty}=0\text{,}\label{5}%
\end{equation}
then for any $\phi \in C_{\divideontimes}[0,\infty)$ we have
\begin{equation}
\lim_{s\rightarrow \infty}\left \Vert
\mathcal{F}%
_{P,\Re}^{s}\phi-\phi \right \Vert _{\infty}=0\label{4}%
\end{equation}

\end{theorem}

\begin{proof}
Suppose that $\phi \in C_{\divideontimes}[0,\infty)$. Then there exists a
constant $H>0$ such that $\left \vert \phi(\xi)\right \vert \leq H$ for any
$\xi \in \lbrack0,\infty).$ So we get
\begin{equation}
\left \vert \phi \left(  t\right)  -\phi \left(  \xi \right)  \right \vert
\leq2H\label{7a}%
\end{equation}
for $\xi \in \lbrack0,\infty).$ It is easy to see that for a given
$\varepsilon>0$ there is a $\eta>0$ such that%
\[
\left \vert \phi \left(  t\right)  -\phi \left(  \xi \right)  \right \vert
<\varepsilon
\]
whenever $\left \vert e^{-t}-e^{-\xi}\right \vert <\eta$ for all $\xi \in
\lbrack0,\infty).$ Thus, we reach
\begin{equation}
\left \vert \phi \left(  t\right)  -\phi \left(  \xi \right)  \right \vert
\leq \varepsilon+\frac{2H}{\eta^{2}}\left(  e^{-t}-e^{-\xi}\right)
^{2}.\label{7b}%
\end{equation}
On the other hand, we can write
\begin{align}
\left \vert \Re_{m}\left(  \phi \left(  t\right)  ;\xi \right)  -\phi \left(
\xi \right)  \right \vert  &  \leq \left(  \Re_{m}\left(  \left \vert \phi \left(
t\right)  -\phi \left(  \xi \right)  \right \vert ;\xi \right)  \right)
\nonumber \\
&  +\left \Vert \phi \right \Vert _{\infty}\left \vert \left(  \Re_{m}\left(
1;\xi \right)  -1\right)  \right \vert .\label{7d}%
\end{align}
Applying the operators $\left \{  \Re_{m}\right \}  $ to $\left(  \ref{7b}%
\right)  $, and using the linearity and positivity of $\left \{  \Re
_{m}\right \}  $ with $\left(  \ref{7d}\right)  $, we obtain
\begin{align*}
\left \vert
\mathcal{F}%
_{P,\Re}^{s}\left(  \phi \left(  t\right)  ;\xi \right)  -\phi \left(
\xi \right)  \right \vert  &  \leq \varepsilon+\left(  \varepsilon+H+\frac
{2H}{\eta^{2}}m\right)  \left \vert
{\displaystyle \int \limits_{0}^{\infty}}
\mathcal{F}%
(s,y)\left(  \mathcal{V}_{P,\Re}^{y}\phi_{0}\right)  \left(  y\right)
dy-\phi_{0}(\xi)\right \vert \\
&  +\frac{4H}{\eta^{2}}\left \vert
{\displaystyle \int \limits_{0}^{\infty}}
\mathcal{F}%
(s,y)\left(  \mathcal{V}_{P,\Re}^{y}\phi_{1}\right)  \left(  y\right)
dy-\phi_{1}(\xi)\right \vert \\
&  +\frac{2H}{\eta^{2}}\left \vert
{\displaystyle \int \limits_{0}^{\infty}}
\mathcal{F}%
(s,y)\left(  \mathcal{V}_{P,\Re}^{y}\phi_{2}\right)  \left(  y\right)
dy-\phi_{2}(\xi)\right \vert
\end{align*}
yielding
\[
\lim_{s\rightarrow \infty}\left \Vert
\mathcal{F}%
_{P,\Re}^{s}\phi-\phi \right \Vert _{\infty}=0.
\]

\end{proof}

The regularity of $P$ and $\mathcal{F}$ guarantee that ($\ref{5}$) holds
whenever the conditions of the classical Korovkin theorem are met. However,
since Example $\ref{example2}$ shows that the converse is not generally true,
Theorem $\ref{theorem1}$ is demonstrated to be strictly stronger than both the
classical Korovkin theorem and Theorem $\ref{theorem2}$.

\begin{example}
\label{example2}Let $\alpha=\left \{  \alpha_{m}\right \}  $ denote a sequence
that is Abel convergent to $1$ but fails to be Borel convergent, and take into
account the Sz\'{a}sz--Mirakjan operators $\left \{  S_{m}\right \}  $ defined
in (\ref{szasz}). We introduce the positive linear operators $\Re_{m}$ by
\[
\Re_{m}=\alpha_{m}S_{m+1}%
\]
for any $m=0,1,...$ Since $\alpha$ does not converge in the classical sense,
(\ref{SZASZ1}) implies that $\left \{  \Re_{m}\phi_{0}\right \}  $ is not
uniformly convergent. Consequently, the classical Korovkin theorem cannot be
applied to the sequence $\left \{  \Re_{m}\right \}  $. Furthermore, because
$\alpha$ is not Borel convergence, (\ref{SZASZ1}) again ensures that $\left \{
\Re_{m}\phi_{0}\right \}  $ does not converge under Borel convergence.
Therefore, Theorem \ref{theorem2} is not valid in this context. On the other
hand, we consider the case where $\mathcal{F}$ is the Abel integral summability method, defined by
\[
\mathcal{F}(s,y)=\frac{1}{s}e^{-y/s},
\]
as introduced in \cite{connor}, and $P$ is the Borel power series method $B$.
In this case, the general operator defined in (\ref{22}) simplifies to
\[
(\mathcal{F}_{B,\Re}^{s}\phi)(\xi)=\frac{1}{s}{\displaystyle \int
\limits_{0}^{\infty}}e^{-y/s}e^{-y}\left(  \sum_{m=0}^{\infty}\Re_{m}%
(\phi(y);\xi)\frac{y^{m}}{m!}\right)  dy.
\]
By interchanging the integral and the sum, justified by the Lebesgue Monotone
Convergence Theorem, and using the Gamma function, this operator is further
reduced to a standard series-based form
\begin{equation}
(\mathcal{F}_{B,\Re}^{s}\phi)(\xi)=\frac{1}{s+1}\sum_{m=0}^{\infty}\left(
\frac{s}{s+1}\right)  ^{m}\Re_{m}(\phi(y);\xi).\label{24}%
\end{equation}
Now, by performing the substitution $y=\dfrac{s}{s+1}$, we get: $0<y<1$,
$y\rightarrow1^{-}$ as $s\rightarrow \infty$ and
\begin{equation}
\frac{1}{s+1}\sum_{m=0}^{\infty}\left(  \frac{s}{s+1}\right)  ^{m}\Re_{m}%
(\phi;\xi)=(1-y)\sum_{m=0}^{\infty}y^{m}\Re_{m}(\phi;\xi)\text{.}\label{7e}%
\end{equation}
From (\ref{SZASZ1}) and (\ref{7e}), we have%
\begin{align*}
\left \vert (\mathcal{F}_{B,\Re}^{s}\phi_{0})(\xi)-\phi_{0}\right \vert  &
=\left \vert (1-y)\sum_{m=0}^{\infty}\alpha_{m}y^{m}S_{m+1}(\phi_{0}%
(t);\xi)-\phi_{0}\right \vert \\
&  \leq(1-y)\sum_{m=0}^{\infty}y^{m}\left \vert \alpha_{m}-1\right \vert .
\end{align*}
By the hypothesis, we get%
\[
\lim_{s\rightarrow \infty}\left \Vert \mathcal{F}_{B,\Re}^{s}\phi_{0}-\phi
_{0}\right \Vert _{\infty}=0.
\]
For $\nu=1,$ by considering (\ref{szasz-2}), we can write%
\begin{align}
\left \vert \left(  \mathcal{F}_{B,\Re}^{s}\phi_{1}\right)  \left(  \xi \right)
-\phi_{1}\left(  \xi \right)  \right \vert  &  =\left \vert (1-y)\sum
_{m=0}^{\infty}y^{m}\alpha_{m}e^{-\xi \left(  m+1\right)  (1-e^{-\frac{1}{m+1}%
})}-e^{-\xi}\right \vert \nonumber \\
&  \leq \left \vert (1-y)\sum_{m=0}^{\infty}y^{m}\alpha_{m}e^{-\xi \left(
m+1\right)  (1-e^{-\frac{1}{m+1}})}-\alpha_{m}e^{-\xi}\right \vert \nonumber \\
&  +(1-y)\left \vert \sum_{m=0}^{\infty}y^{m}e^{-\xi}\left(  \alpha
_{m}-1\right)  \right \vert \nonumber \\
&  \leq(1-y)\sum_{m=0}^{\infty}y^{m}\alpha_{m}\left \vert e^{-\xi \left(
m+1\right)  (1-e^{-\frac{1}{m+1}})}-e^{-\xi}\right \vert \nonumber \\
&  +e^{-\xi}(1-y)\sum_{m=0}^{\infty}y^{m}\left \vert \alpha_{m}-1\right \vert
.\label{7f}%
\end{align}
If we define $\beta_{m}=m(1-e^{-\frac{1}{m}})$ for $m=1,2,...$, and take into
account the inequality (see; \cite{holhos2018} )%
\begin{align}
e^{-\xi \beta_{m}}-e^{-\xi} &  \leq \left(  \frac{1-\beta_{m}}{2}\right)
\left(  \xi e^{-\xi \beta_{m}}+\xi e^{-\xi}\right)  \nonumber \\
&  \leq \frac{1-\beta_{m}^{2}}{2e\beta_{m}},\label{7g}%
\end{align}
then we have%
\begin{align*}
\left \vert \left(  \mathcal{F}_{B,\Re}^{s}\phi_{1}\right)  \left(  \xi \right)
-\phi_{1}\left(  \xi \right)  \right \vert  &  \leq(1-y)\left \vert \sum
_{m=0}^{\infty}y^{m}\alpha_{m}\frac{1-\beta_{m+1}^{2}}{2e\beta_{m+1}%
}\right \vert \\
&  +(1-y)\sum_{m=0}^{\infty}y^{m}e^{-\xi}\left \vert \alpha_{m}-1\right \vert .
\end{align*}
By the hypothesis, we reach to%
\[
\lim_{s\rightarrow \infty}\left \Vert \mathcal{F}_{p,\Re}^{s}\phi_{1}-\phi
_{1}\right \Vert _{\infty}=0.
\]
For $\nu=2,$ by using (\ref{szasz-3})and (\ref{7e}), we can write%
\begin{align}
\left \vert \left(  \mathcal{F}_{B,\Re}^{s}\phi_{2}\right)  \left(  \xi \right)
-\phi_{2}\left(  \xi \right)  \right \vert  &  =\left \vert (1-y)\sum
_{m=0}^{\infty}y^{m}\alpha_{m}e^{-\xi \left(  m+1\right)  (1-e^{-\frac{2}{m+1}%
})}-e^{-2\xi}\right \vert \nonumber \\
&  \leq \left \vert (1-y)\sum_{m=0}^{\infty}y^{m}\alpha_{m}e^{-2\xi \frac{m+1}%
{2}(1-e^{-\frac{2}{m+1}})}-(1-y)\sum_{m=0}^{\infty}y^{m}\alpha_{m}e^{-2\xi
}\right \vert \nonumber \\
&  +\left \vert (1-y)\sum_{m=0}^{\infty}y^{m}e^{-2\xi}\left(  \alpha
_{m}-1\right)  \right \vert \nonumber \\
&  \leq(1-y)\sum_{m=1}^{\infty}y^{m}\alpha_{m}\left \vert e^{-2\xi \frac{m+1}%
{2}(1-e^{-\frac{2}{m+1}})}-e^{-2\xi}\right \vert \nonumber \\
&  +e^{-2\xi}(1-y)\sum_{m=0}^{\infty}y^{m}\left \vert \alpha_{m}-1\right \vert
.\label{7h}%
\end{align}
Taking $\eth _{m}=\dfrac{m}{2}(1-e^{-\frac{2}{m}}),$ and considering the
inequality given in \cite{holhos2018}, we have%
\begin{align}
e^{-2\xi \eth _{m}}-e^{-2\xi} &  \leq \left(  \frac{2-\eth _{m}}{2}\right)
\left(  \xi e^{-\xi \eth _{m}}+\xi e^{-2\xi}\right)  \nonumber \\
&  \leq \frac{4-\eth _{m}^{2}}{4e\eth _{m}}.\label{7k}%
\end{align}
By using (\ref{7k}) in (\ref{7h}), we get%
\begin{align*}
\left \vert \left(  \mathcal{F}_{B,\Re}^{s}\phi_{2}\right)  \left(  \xi \right)
-\phi_{2}\left(  \xi \right)  \right \vert  &  \leq(1-y)\left \vert \sum
_{m=0}^{\infty}y^{m}\alpha_{m}\left(  \frac{4-\eth _{m+1}^{2}}{4e\eth _{m+1}%
}\right)  \right \vert \\
&  +(1-y)\sum_{m=0}^{\infty}y^{m}e^{-2\xi}\left \vert \alpha_{m}-1\right \vert ,
\end{align*}
which implies%
\[
\lim_{s\rightarrow \infty}\left \Vert \mathcal{F}_{B,\Re}^{s}\phi_{2}-\phi
_{2}\right \Vert _{\infty}=0.
\]
Since conditions (\ref{5}) are satisfied, from Theorem \ref{theorem1} we have
the following.
\[
\lim_{s\rightarrow \infty}\left \Vert \mathcal{F}_{B,\Re}^{s}\phi-\phi
\right \Vert _{\infty}=0,
\]
for any $\phi \in C_{\divideontimes}[0,\infty)$.
\end{example}

In Example~\ref{exmp3}, we demonstrate the uniform convergence of 
\( \mathcal{F}_{B,\Re}^{s}\phi \) to \( \phi \in C_{\divideontimes}[0,\infty) \) 
as \( s \to \infty \) for a particular choice of a sequence of positive linear operators. 
It should be noted that the sequence of positive linear operators given in this example 
does not satisfy the conditions of the classical Korovkin theorem.

\begin{example}
\label{exmp3} Let us reconsider the Abel integral summability method \( \mathcal{F} \), 
the Borel power series method \( B \), and the sequence of 
Sz\'{a}sz--Mirakjan operators \( (S_{m}) \). 
We then define a new sequence of positive linear operators \( \left\{\Re_{m}\right\} \) by
\begin{equation}
\Re_{m}\phi=\left \{
\begin{array}
[c]{ll}%
0 & ,\text{ if }m\text{ is perfect square or }m=0\\
S_{m+1} & ,\text{ otherwise.}%
\end{array}
\right.  \label{23}%
\end{equation}
From (\ref{SZASZ1}) and (\ref{24}) we have%
\begin{equation}
(%
\mathcal{F}%
_{B,\Re}^{s}\phi_{0})(\xi)=\frac{1}{s+1}\left(  \sum_{m=0}^{\infty}\left(
\frac{s}{s+1}\right)  ^{m}-\sum_{k=0}^{\infty}\left(  \frac{s}{s+1}\right)
^{k^{2}}\right)  \text{.}\label{14}%
\end{equation}
From (3.13) in \cite{ulucay-un}, we can write
\begin{equation}
\lim_{s\rightarrow \infty}\frac{1}{s+1}\sum_{k=0}^{\infty}\left(  \frac{s}%
{s+1}\right)  ^{k^{2}}=\lim_{y\rightarrow1^{-}}\left(  1-y\right)  \sum
_{k=0}^{\infty}y^{k^{2}}=0.\label{15}%
\end{equation}
for $y=\dfrac{s}{s+1}$, such that $0<y<1.$ Thus, from (\ref{14}) and
(\ref{15}) we have%
\[
\left \vert
\mathcal{F}%
_{B,\Re}^{s}\phi_{0}-\phi_{0}\right \vert =\left \vert (1-y)\sum_{m=0}^{\infty
}y^{m}-\left(  1-y\right)  \sum_{k=0}^{\infty}y^{k^{2}}-1\right \vert ,
\]
which leads to%
\[
\lim_{s\rightarrow \infty}\left \Vert
\mathcal{F}%
_{B,\Re}^{s}\phi_{0}-\phi_{0}\right \Vert _{\infty}=0\text{.}%
\]
For $\nu=1$, by using (\ref{szasz-2}) and (\ref{7e}), we obtain$\ $%
\begin{align*}
(%
\mathcal{F}%
_{B,\Re}^{s}\phi_{1})(\xi) &  =(1-y)\sum_{m=0}^{\infty}y^{m}\Re_{m}(\phi
_{1}\left(  y\right)  ;\xi)-\left(  1-y\right)  \sum_{k=0}^{\infty}y^{k^{2}%
}\Re_{m}(\phi_{1}\left(  y\right)  ;\xi)\\
&  =(1-y)\sum_{m=0}^{\infty}y^{m}e^{-\xi \beta_{m+1}}-\left(  1-y\right)
\sum_{k=0}^{\infty}y^{k^{2}}e^{-\xi \beta_{k^{2}+1}}.
\end{align*}
Considering (\ref{7g}) for $\xi \in \lbrack0,\infty)$, we have%
\begin{align}
\left \vert \left(
\mathcal{F}%
_{B,\Re}^{s}\phi_{1}\right)  \left(  \xi \right)  -\phi_{1}\left(  \xi \right)
\right \vert  &  =\left \vert (1-y)\sum_{m=0}^{\infty}y^{m}e^{-\xi \beta_{m+1}%
}-\left(  1-y\right)  \sum_{k=0}^{\infty}y^{k^{2}}e^{-\xi \beta_{k^{2}+1}%
}-e^{-\xi}\right \vert \nonumber \\
&  =\left \vert (1-y)\sum_{m=0}^{\infty}y^{m}\left(  e^{-\xi \beta_{m+1}%
}-e^{-\xi}\right)  -\left(  1-y\right)  \sum_{k=0}^{\infty}y^{k^{2}}\left(
e^{-\xi \beta_{k^{2}+1}}-e^{-\xi}\right)  \right \vert \nonumber \\
&  +\left \vert e^{-\xi}\left \{  (1-y)\sum_{m=0}^{\infty}y^{m}-\left(
1-y\right)  \sum_{k=0}^{\infty}y^{k^{2}}-1\right \}  \right \vert \nonumber \\
&  \leq \left \vert (1-y)\sum_{m=0}^{\infty}y^{m}\left(  e^{-\xi \beta_{m+1}%
}-e^{-\xi}\right)  \right \vert +\left \vert \left(  1-y\right)  \sum
_{k=0}^{\infty}y^{k^{2}}\left(  e^{-\xi \beta_{k^{2}+1}}-e^{-\xi}\right)
\right \vert \nonumber \\
&  \leq \left \vert (1-y)\sum_{m=0}^{\infty}y^{m}c_{m+1}\right \vert +\left \vert
\left(  1-y\right)  \sum_{k=0}^{\infty}y^{k^{2}}\left(  e^{-\xi \beta_{k^{2}%
+1}}-e^{-\xi}\right)  \right \vert \nonumber \\
&  +\left \vert (1-y)\sum_{m=0}^{\infty}y^{m}-\left(  1-y\right)  \sum
_{k=0}^{\infty}y^{k^{2}}-1\right \vert ,\label{16}%
\end{align}
where%
\[
c_{m}=\left \{
\begin{array}
[c]{ll}%
\frac{1-\beta_{m}^{2}}{2e\beta_{m}}, & \text{if }m\text{ is perfect square}\\
0, & \text{otherwise,}%
\end{array}
\right.
\]
which is Abel convergent to zero since it is null. Thus, we have%
\[
\lim_{s\rightarrow \infty}\left \Vert
\mathcal{F}%
_{B,\Re}^{s}\phi_{1}-\phi_{1}\right \Vert _{\infty}=0\text{.}%
\]
For $\nu=2,$ considering (\ref{szasz-3}) and (\ref{24}), we have%
\begin{align*}
(%
\mathcal{F}%
_{B,\Re}^{s}\phi_{2})(\xi) &  =(1-y)\sum_{m=0}^{\infty}y^{m}\Re_{m}(\phi
_{2}\left(  y\right)  ;\xi)-\left(  1-y\right)  \sum_{k=0}^{\infty}y^{k^{2}%
}\Re_{m}(\phi_{2}\left(  y\right)  ;\xi)\\
&  =(1-y)\sum_{m=0}^{\infty}y^{m}e^{-\xi m(1-e^{-\frac{2}{m}})}-\left(
1-y\right)  \sum_{k=0}^{\infty}y^{k^{2}}e^{-\xi m(1-e^{-\frac{2}{m}})}%
\end{align*}
which leads to%
\begin{align}
\left \vert (%
\mathcal{F}%
_{B,\Re}^{s}\phi_{2})(\xi)-\phi_{2}(\xi)\right \vert  &  =\left \vert
(1-y)\sum_{m=0}^{\infty}y^{m}e^{-\xi \eth _{m+1}}-\left(  1-y\right)
\sum_{k=0}^{\infty}y^{k^{2}}e^{-\xi \eth _{k^{2}+1}}-e^{-2\xi}\right \vert
\nonumber \\
&  \pm e^{-2\xi}\left[  (1-y)\sum_{m=0}^{\infty}y^{m}-\left(  1-y\right)
\sum_{k=0}^{\infty}y^{k^{2}}\right]  \nonumber \\
&  =\left \vert (1-y)\sum_{m=0}^{\infty}y^{m}\left(  e^{-2\xi \eth _{m+1}%
}-e^{-2\xi}\right)  -\left(  1-y\right)  \sum_{k=0}^{\infty}y^{k^{2}}\left(
e^{-\xi \eth _{k^{2}+1}}-e^{-2\xi}\right)  \right \vert \nonumber \\
&  +\left \vert e^{-2\xi}\left \{  (1-y)\sum_{m=0}^{\infty}y^{m}-\left(
1-y\right)  \sum_{k=0}^{\infty}y^{k^{2}}-1\right \}  \right \vert \nonumber \\
&  \leq \left \vert (1-y)\sum_{m=0}^{\infty}y^{m}\left(  e^{-2\xi \eth _{m+1}%
}-e^{-2\xi}\right)  \right \vert +\left \vert \left(  1-y\right)  \sum
_{k=0}^{\infty}y^{k^{2}}\left(  e^{-\xi \eth _{k^{2}+1}}-e^{-2\xi}\right)
\right \vert \nonumber \\
&  \leq \left \vert (1-y)\sum_{m=0}^{\infty}y^{m}\frac{4-\eth _{m+1}^{2}%
}{4e\eth _{m+1}}\right \vert +\left \vert \left(  1-y\right)  \sum_{k=0}%
^{\infty}y^{k^{2}}\left(  e^{-\xi \eth _{k^{2}+1}}-e^{-2\xi}\right)
\right \vert \nonumber \\
&  +\left \vert (1-y)\sum_{m=0}^{\infty}y^{m}-\left(  1-y\right)  \sum
_{k=0}^{\infty}y^{k^{2}}-1\right \vert ,\label{18}%
\end{align}
Using similar modifications as in \cite{ulucay-un} we have from (\ref{18})
that%
\[
\lim_{s\rightarrow \infty}\left \Vert
\mathcal{F}%
_{B,\Re}^{s}\phi_{2}-\phi_{2}\right \Vert _{\infty}=0\text{.}%
\]
So, the hypotheses of Theorem \ref{theorem1} are satisfied and the proof is ended.
\end{example}

\color{black}

\section{Rates of the Convergence}\label{sec5}

In this section, we compute the rates of the convergences by means of the
modulus of continuity via summability methods given in previous two sections.

The modulus of continuity of $\omega \left(  \phi,\delta \right)  $ is defined
by%
\begin{equation}
\omega \left(  \phi,\delta \right)  =\sup_{\substack{\left \vert t-\xi \right \vert
\leq \delta \\0\leq t,\xi\leq a   }}\left \vert \phi \left(  t\right)
-\phi \left(  \xi \right)  \right \vert .\label{sur-mod}%
\end{equation}
By using Korovkin subset \{$1,e^{-\xi},$ $e^{-2\xi}$\}, Holho\c{s}
\cite{holhos} introduced the modulus of continuity given by
\begin{equation}
\hat{\omega}(\phi,\delta)=\sup_{\substack{t,\xi \geq0\\ \left \vert
_{e^{-t}-e^{-\xi}}\right \vert \leq \delta}}\left \vert \phi(t)-\phi
(\xi)\right \vert \label{ex-sur-mod}%
\end{equation}
for $\delta \geq0$ and every function $\phi \in C_{\divideontimes}[0,\infty)$.
It has a relationship with the usual modulus of continuity given in (\ref{sur-mod}) ( see;
\cite{holhos}). We obtain the property
\begin{equation}
\left \vert g(u)-g(v)\right \vert \leq(1+\dfrac{\left(  u-v\right)  ^{2}}%
{\delta^{2}})\hat{\omega}\left(  g,\delta \right)  .\label{5g}%
\end{equation}
By taking $u=e^{-t},$ $v=e^{-\xi},$ we have%
\begin{equation}
\left \vert \phi \left(  t\right)  -\phi \left(  \xi \right)  \right \vert
\leq \left(  1+\dfrac{\left(  e^{-t}-e^{-\xi}\right)  ^{2}}{\delta^{2}}\right)
\hat{\omega}(\phi,\delta).\label{5h}%
\end{equation}

The following theorem gives the rate of convergence via the power series method.

\begin{theorem}
\label{rate1} Let $\left \{  \Re_{m}\right \}  $ be a sequence of linear
positive operators from $C_{\divideontimes}[0,\infty)$ to $B_{\divideontimes
}[0,\infty)\ $ such that $%
{\displaystyle \sum \limits_{m=0}^{\infty}}
\left \Vert \Re_{m}\phi_{0}\right \Vert _{\infty}y^{m}<\infty$ for any $0<y<R$.
Assume that $\varsigma,$ $\lambda$ are positive real functions defined on
$[0,\infty)$. If \newline \textbf{i)} $\dfrac{1}{p(y)}\left \Vert
\displaystyle \sum_{m=0}^{\infty}(\Re_{m}\phi_{0}-\phi_{0})y^{m}\right \Vert
_{\infty}=o(\varsigma(y)),$ as $0<y\rightarrow R^{-}$,\newline \textbf{ii)
}$\hat{\omega}(\phi,\delta)=o(\lambda(y))$, as $0<y\rightarrow R^{-}$, where
$\delta=\sqrt{\left \Vert \Re_{m}\left(  \mu\right)  \right \Vert
_{\infty}}$ with $\mu(t)=\displaystyle \sup_{\xi \geq0}\left(
e^{-t}-e^{-\xi}\right)  ^{2}.$ Then, for any $\phi \in C_{\divideontimes
}[0,\infty)$ we have
\[
\frac{1}{p(y)}\left \Vert \left(  \Re_{m}\phi-\phi \right)  \frac{y^{m}}%
{m!}\right \Vert _{\infty}=o(\varphi(y))
\]
as $0<y\rightarrow R^{-}$, where $\varphi(y)=\displaystyle \max_{y>0}\left \{
\varsigma(y),\text{ }\lambda(y)\right \}  .$
\end{theorem}

\begin{proof}
For any $\phi \in C_{\divideontimes}[0,\infty)$, any $\xi,t\in \lbrack0,\infty
)$, and any $\delta>0$ we can write%
\begin{align*}
\left \vert \Re_{m}\left(  \phi(t);\xi \right)  -\phi \left(  \xi \right)
\right \vert  &  \leq \Re_{m}(\left \vert \phi \left(  t\right)  -\phi \left(
\xi \right)  \right \vert ;\xi)+\left \vert \phi(\xi)\right \vert \left \vert
\Re_{m}\left(  \phi_{0}(t);\xi \right)  -\phi_{0}(\xi))\right \vert \\
&  \leq \Re_{m}\left(  \left(  1+\dfrac{\left(  e^{-t}-e^{-\xi}\right)  ^{2}%
}{\delta^{2}}\right)  \hat{\omega}(\phi,\delta);\xi \right)  +M\left \vert
\Re_{m}\left(  \phi_{0}(t);\xi \right)  -\phi_{0}(\xi))\right \vert \\
&  \leq \hat{\omega}(\phi,\delta)\Re_{m}\left(  \left(  1+\dfrac{\left(
e^{-t}-e^{-\xi}\right)  ^{2}}{\delta^{2}}\right)  ;\xi \right)  +M\left \vert
\Re_{m}\left(  \phi_{0}(t);\xi \right)  -\phi_{0}(\xi))\right \vert \\
&  \leq \hat{\omega}(\phi,\delta)\left \vert \Re_{m}\left(  \phi_{0}%
(t);\xi \right)  -\phi_{0}(\xi))\right \vert +\frac{1}{\delta^{2}}\hat{\omega
}(\phi,\delta)\Re_{m}\left(  \mu \left(  t\right)  ;\xi \right)  +\hat{\omega
}(\phi,\delta)\\
&  +M\left \vert \Re_{m}\left(  \phi_{0}(t);\xi \right)  -\phi_{0}%
(\xi))\right \vert.
\end{align*}
Since $\delta=\sqrt{\left \Vert \Re_{m}\left(  \mu\right)
\right \Vert _{\infty}}$ we get
\begin{align*}
\frac{1}{p(y)}\left \Vert \sum \limits_{m=0}^{\infty}\left(  \Re_{m}\phi
-\phi \right)  \frac{y^{m}}{m!}\right \Vert _{\infty} &  \leq \hat{\omega}%
(\phi,\delta)\frac{1}{p(y)}_{\infty}\left \Vert \sum \limits_{m=0}^{\infty}%
(\Re_{m}\left(  \phi_{0}\right)  -\phi_{0})\frac{y^{m}}{m!}\right \Vert
_{\infty}\\
&  +2\hat{\omega}(\phi,\delta)+M\frac{1}{p(y)}_{\infty}\left \Vert
\sum \limits_{m=0}^{\infty}((\Re_{m}\left(  \phi_{0}\right)  -\phi_{0}%
)\frac{y^{m}}{m!}\right \Vert _{\infty},
\end{align*}
which completes the proof.
\end{proof}

The following theorem gives the rate of convergence via the the integral summability method.

\begin{theorem}
\label{rate2} Let $\mathcal{F}
$ be a non-negative regular integral summability method, $P$ be a non-polynomial Borel-type power series method and let $\left\{  \Re_{m}\right\}
$ be a sequence of positive linear operators from $C_{\divideontimes}%
[0,\infty)$ to $B_{\divideontimes}[0,\infty)$ that satisfies (\ref{1}). Assume
that $\pi_{1},$ $\pi_{2}$ are positive non-increasing functions defined
on $[0,\infty).$ If\newline \textbf{i)} $\left \Vert
\mathcal{F}
_{p,\Re}^{s}\left(  \phi_{0}\right)  -\phi_{0}\right \Vert _{\infty}=o(\pi
_{1}(s))$ as $s\rightarrow \infty$,\newline \textbf{ii)} $\hat{\omega}(\phi
,\delta)=o(\pi_{2}(s))$, as $s\rightarrow \infty$, where $\delta=\sqrt{\left \Vert
\Re_{m}\left(  \mu\right)  \right \Vert _{\infty}}$ with
$\mu(t)=\displaystyle \sup_{\xi \geq0}\left(  e^{-t}-e^{-\xi
}\right)  ^{2}.$ Then, for any $\phi \in$ $C_{\divideontimes}[0,\infty)$%
\[
\left \Vert
\mathcal{F}%
_{p,\Re}^{s}\phi-\phi \right \Vert _{\infty}=o(\pi(s))
\]
here $\pi(s)=\displaystyle\max_{s\geq0}\left \{  \pi_{1}(s),\pi_{2}(s)\right \}  $.
\end{theorem}

\begin{proof}
For any $\phi \in C_{\divideontimes}[0,\infty)$, any $\xi,t\in \lbrack a,b]$, and
any $\delta>0$ we can write%
\begin{align*}
\left \vert ((%
\mathcal{F}%
_{P,\Re}^{s}\phi);\xi)-\phi \left(  \xi \right)  \right \vert  &  \leq%
\mathcal{F}%
_{P,\Re}^{s}(\left \vert \phi \left(  t\right)  -\phi \left(  \xi \right)
\right \vert ;\xi)+\left \vert \phi(\xi)\right \vert \left \vert
\mathcal{F}%
_{P,\Re}^{s}\left(  \phi_{0}(t);\xi \right)  -\phi_{0}(\xi))\right \vert \\
&  \leq%
\mathcal{F}%
_{P,\Re}^{s}\left(  \left(  1+\dfrac{\left(  e^{-t}-e^{-\xi}\right)  ^{2}%
}{\delta^{2}}\right)  \hat{\omega}(\phi,\delta);\xi \right)  +M\left \vert
\mathcal{F}%
_{P,\Re}^{s}\left(  \phi_{0}(t);\xi \right)  -\phi_{0}(\xi))\right \vert \\
&  \leq \hat{\omega}(\phi,\delta)%
\mathcal{F}%
_{P,\Re}^{s}\left(  \left(  1+\dfrac{\left(  e^{-t}-e^{-\xi}\right)  ^{2}%
}{\delta^{2}}\right)  ;\xi \right)  +M\left \vert
\mathcal{F}%
_{P,\Re}^{s}\left(  \phi_{0}(t);\xi \right)  -\phi_{0}(\xi))\right \vert \\
&  \leq \hat{\omega}(\phi,\delta)\left \vert
\mathcal{F}%
_{P,\Re}^{s}\left(  \phi_{0}(t);\xi \right)  -\phi_{0}(\xi)\right \vert
+\frac{1}{\delta^{2}}\hat{\omega}(\phi,\delta)%
\mathcal{F}%
_{P,\Re}^{s}\left(  \mu \left(  t\right)  ;\xi \right)  +\hat{\omega}%
(\phi,\delta)\\
&  +M\left \vert
\mathcal{F}%
_{P,\Re}^{s}\left(  \phi_{0}(t);\xi \right)  -\phi_{0}(\xi))\right \vert.
\end{align*}
As $\delta=\sqrt{\left \Vert
\mathcal{F}%
_{P,\Re}^{s}\left(  \mu\right)  \right \Vert _{\infty}}$ we
get
\begin{align*}
\frac{1}{p(y)}\left \Vert \sum \limits_{m=0}^{\infty}\left(  \Re_{m}\phi
-\phi \right)  \frac{y^{m}}{m!}\right \Vert _{\infty} &  \leq \hat{\omega}%
(\phi,\delta)\frac{1}{p(y)}_{\infty}\left \Vert \sum \limits_{m=0}^{\infty}%
(\Re_{m}\left(  \phi_{0}\right)  -\phi_{0})\frac{y^{m}}{m!}\right \Vert
_{\infty}\\
&  +2\hat{\omega}(\phi,\delta)+M\frac{1}{p(y)}_{\infty}\left \Vert
\sum \limits_{m=0}^{\infty}((\Re_{m}\left(  \phi_{0}\right)  -\phi_{0}%
)\frac{y^{m}}{m!}\right \Vert _{\infty},
\end{align*}
which implies%
\[
\left \Vert
\mathcal{F}%
_{P,\Re}^{s}\phi-\phi \right \Vert _{\infty}\leq \hat{\omega}(\phi,\delta
)\left \Vert
\mathcal{F}%
_{P,\Re}^{s}\phi_{0}-\phi_{0}\right \Vert _{\infty}+2\hat{\omega}(\phi
,\delta)+M\left \Vert
\mathcal{F}%
_{P,\Re}^{s}\phi_{0}-\phi_{0}\right \Vert _{\infty}.
\]
So, the proof is ended.
\end{proof}

\section{Conclusion}\label{sec6}
In this study, we have established Korovkin-type approximation theorems that preserve exponential functions by employing power series convergence and its special case, the Borel method. Moreover, by introducing approximation through Borel-type power series convergence via integral summability, we have proposed an alternative framework that remains valid in cases where classical or ordinary Borel convergence fails. This formulation enhances the applicability of Korovkin-type results on infinite intervals and provides stronger convergence behavior supported by quantitative estimates through the modulus of continuity. The theoretical results were further validated by illustrative examples, confirming the efficiency of the proposed approach. Future research may explore the extension of the present results to broader classes of summability methods, such as matrix or fractional-type summability. It would also be of interest to study analogous Korovkin-type theorems for other classes of functions and to investigate possible connections with statistical convergence, or operators defined on weighted function spaces.

\end{document}